\documentclass[reqno,12pt]{amsart}
\usepackage[T1]{fontenc}
\usepackage[utf8]{inputenc}
\usepackage[english]{babel}
\usepackage{amssymb,amsmath,amsthm,amscd,latexsym,amsfonts,xcolor,enumerate,hyperref,comment,longtable,cleveref}

\makeatletter
\@namedef{subjclassname@2020}{%
  \textup{2020} Mathematics Subject Classification}
\makeatother

\usepackage{times}
\usepackage{cite}
\usepackage{pdflscape}
\usepackage[mathcal]{euscript}
\usepackage{tikz}
\usepackage{hyperref}
\usepackage{cancel}
\usepackage{stmaryrd}
\usepackage{dsfont}
 \usepackage{array,multirow,graphicx}

\numberwithin{equation}{section}

\usetikzlibrary{arrows}

\DeclareMathOperator{\gr}{gr}

\usepackage[a4paper,top=3cm, bottom=3cm, left=3cm, right=3cm]{geometry}

\newtheorem{lemma}[subsection]{Lemma}
\newtheorem{teo}[subsection]{Theorem}

\theoremstyle{definition}
\newtheorem{de}[subsection]{Definition}

\theoremstyle{remark}

\usepackage{stmaryrd}

\begin{document}

\title[Naturally graded p-filiform associative algebras]
{Naturally graded p-filiform associative algebras}

	\author{I.A.Karimjanov}
\address{Department of Mathematics, Andijan State University, 129, Universitet Street, Andijan,
170100, Uzbekistan,
\newline and\newline
V.I.Romanovskiy Institute of Mathematics, Uzbekistan Academy of Sciences, Univesity Street, 9, Olmazor district, Tashkent, 100174, Uzbekistan,
\newline
		and\newline
		Saint-Petersburg State University, Saint-Petersburg, Russia.}

\email{iqboli@gmail.com}

\thanks{This work was partially supported by RSF 22-71-10001.}

\begin{abstract} In the paper we describe $n$-dimensional naturally graded nilpotent associative algebras with the characteristic sequence $C(\mathcal{A})=(n-p,1,\dots,1)$ as called $p-$filiform algebras over the field of the complex numbers.
\end{abstract}


\maketitle

\section{Introduction}

Associative algebras appeared in the several fields of mathematics, for instance: representation theory, analysis, geometry and combinatorics, mathematical physics and others. The first studies of associative algebras as an algebraic object were started by B.Pierce\cite{Pr}. Later Wedderburn \cite{{wed}} proved that any finite-dimensional associative algebra over a perfect field can be expressed as a semidirect sum of a semisimple subalgebra and its nilpotent radical. It is already well known that the semisimple part can be described from simple ideal and any simple associative algebra isomorphic to $M_n(D)$, the algebra of $n\times n$ matrices with entries from a division algebra $D$. Thus, the study of finite dimensional associative algebras was reduced to the nilpotent ones.

The classification of low dimensional nilpotent associative algebras were done by several authors. Latterly, nilpotent associative algebras of dimension less or equal 4 classified by using the method of central extentions by De Graaf\cite{Graaf}. Recently, in \cite{kar} classified 5-dimensional complex nilpotent associative algebras except two-step nilpotent algebras.

There are a lot of papers devoted to classification of finite-dimensional nilpotent associative algebras with different properties. Actually, $n$-dimensional associative algebras with a maximum nilpotency index as named null-filiform, with nilindex $n$ as named filiform and $n-1$ as named quasi-filiform algebras have been investigated \cite{kar2,kar3}. We remark that $n$-dimensional nilpotent associative algebra has a nilpotency index less or equal $n+1$. Later, the classification of naturally graded nilpotent associative algebras with the characteristic sequence $(n-m,m)$ has been done \cite{kar1}.

The purpose of this paper is to study naturally graded $n$-dimensional associative algebras with characteristic sequence $(n-p,1,\dots,1)$ as named $p$-filiform. Similar results for naturally graded nilpotent Lie, Leibniz and Zinbiel algebras were obtained in the works \cite{Lie, Leib, Zin}.

Throughout the paper, all spaces and algebras are considered over the field of complex numbers. For convenience, we omit the products which are equal to zero.

\section{Preliminaries}\label{S:prel}
For an algebra $\mathcal{A}$ of an arbitrary variety, we consider the series
\[
\mathcal{A}^1=\mathcal{A}, \qquad \ \mathcal{A}^{i+1}=\sum\limits_{k=1}^{i}\mathcal{A}^k \mathcal{A}^{i+1-k}, \qquad i\geq 1.
\]

We say that  an  algebra $\mathcal{A}$ is \emph{nilpotent} if $\mathcal{A}^{i}=0$ for some $i \in \mathbb{N}$. The integer $k$ satisfying $\mathcal{A}^{k}\neq0$ and $\mathcal{A}^{k+1}=0$ is called the \emph{index of nilpotency} or \emph{nilindex} of $\mathcal{A}$.

\begin{de} Given a nilpotent associative algebra $\mathcal{A}$, put
$\mathcal{A}_i=\mathcal{A}^i/\mathcal{A}^{i+1}, \ 1 \leq i\leq k-1$, and $\gr \mathcal{A} = \mathcal{A}_1 \oplus
\mathcal{A}_2\oplus\dots \oplus \mathcal{A}_{k}$. Then $\mathcal{A}_i\mathcal{A}_j\subseteq \mathcal{A}_{i+j}$ and we
obtain the graded algebra $\gr \mathcal{A}$. If $\gr \mathcal{A}$ and $\mathcal{A}$ are isomorphic,
denoted by $\gr \mathcal{A}\cong \mathcal{A}$, we say that the algebra $\mathcal{A}$ is naturally
graded.
\end{de}

For any element $x$ of $\mathcal{A}$ we define the  left multiplication operator as
\[L_x \colon  \mathcal{A} \rightarrow \mathcal{A},  \quad z \mapsto xz, \ z\in\mathcal{A}.\]

Let us $x\in\mathcal{A}\setminus\mathcal{A}^2$ and for the left multiplication operator $L_x$, define the decreasing sequence $C(x)=(n_1,n_2, \dots, n_k)$ that consists of the dimensions of the Jordan blocks of the operator $L_x$. Endow the set of these sequences with the lexicographic order, i.e. $C(x)=(n_1,n_2, \dots, n_k)\leq C(y)=(m_1,m_2, \dots, m_s)$ means that there is an $i\in\mathbb{N}$ such that $n_j=m_j$ for all $j<i$ and $n_i<m_i$.

\begin{de} The sequence $C(\mathcal{A})=\max_{x\in\mathcal{A}\setminus\mathcal{A}^2}C(x)$ is defined to be the characteristic sequence of the algebra $\mathcal{A}$.
\end{de}

We remark that $n-$dimensional nilpotent associative algebra has a nilpotency index less or equal $n$. A nilpotent associative algebras which have a maximum nilindex are one-generated and called null-filiform algebras.

\begin{de}
An $n$-dimensional algebra $\mathcal{A}$ is called null-filiform if $\dim \mathcal{A}^i=(n+ 1)-i,\ 1\leq i\leq n+1$.
\end{de}

There is unique such type associative algebra up to isomorphism for fixed dimension: An arbitrary $n$-dimensional null-filiform associative algebra is isomorphic to the algebra:
\[e_i e_j= e_{i+j}, \quad 2\leq i+j\leq n,\]
where $\{e_1, e_2, \dots, e_n\}$ is a basis of the algebra $\mathcal{A}$.

A nilpotent associative algebras with nilindex $n-1$ which called filiform algebras have been classified in \cite{kar3}.

\begin{de}
An $n$-dimensional algebra is called filiform if $\dim(\mathcal{A}^i)=n-i, \ 2\leq i \leq n$.
\end{de}

\begin{teo}\cite{kar3}
For $n>3$, every $n$-dimensional filiform associative algebra over an algebraically closed field of characteristic zero is isomorphic to one of the following pairwise non-isomorphic algebras with a basis $\{e_1,e_2,\dots,e_n\}$:
\[\begin{array}{llll}
\mu_{1,1}^n: & e_ie_j=e_{i+j},  & & \\
\mu_{1,2}^n: & e_ie_j=e_{i+j},  & e_ne_n=e_{n-1}, & \\
\mu_{1,3}^n: & e_ie_j=e_{i+j},  & e_1e_n=e_{n-1}, & \\
\mu_{1,4}^n: & e_ie_j=e_{i+j},  & e_1e_n=e_{n-1}, & e_ne_n=e_{n-1}
\end{array}\]
where $2\leq i+j\leq n-1$.
\end{teo}

The classification of naturally graded nilpotent non-split associative algebras with nilindex $n-2$ as named quasi-filiform were done in \cite{kar3}.

\begin{teo}\cite{kar3} Let $\mathcal{A}$ be $n$-dimensional $(n>5)$ naturally graded non-split quasi-filiform associative algebra over an algebraically closed field $\mathds{F}$ of characteristic zero. Then it is isomorphic to one of the following pairwise non-isomorphic algebras with a basis $\{e_1,e_2,\dots,e_n\}$:
\[\mu_{2,1}^n:\left\{\begin{array}{l}
  e_ie_j=e_{i+j},  \\
  e_{n-1}e_1=e_n,   \\
  \end{array}\right. \quad \mu_{2,2}^n(\alpha):\left\{\begin{array}{l}
e_ie_j=e_{i+j}, \\
e_1e_{n-1}=e_n,\\
e_{n-1}e_1=\alpha e_n, \\
\end{array}\right. \]
\[\mu_{2,3}^n: \left\{\begin{array}{ll}
e_ie_j=e_{i+j}, \\
e_1e_{n-1}=e_n,\\
e_{n-1}e_1=e_n,\\
e_{n-1}e_{n-1}=e_n, \\
\end{array}\right.
\quad
 \mu_{2,4}^n:\left\{\begin{array}{l}
e_ie_j=e_{i+j}, \\
e_1e_{n-1}=e_n,\\
e_{n-1}e_{n-1}=e_n,\\
\end{array}\right.
\quad
\]
where $\alpha\in\mathds{F}$ and  \, $2\leq i+j\leq n-2$.
\end{teo}

Now we define filiform algebras of degree $p$.

\begin{de}
An $n$-dimensional associative algebra $\mathcal{A}$ is called filiform of degree $p$ if $\dim(\mathcal{A}^i)=n-p+1-i, \ 1\leq i \leq n-p+1$.
\end{de}


\begin{teo}\cite{kar3}.
 Let $\mathcal{A}$ be a naturally graded filiform associative algebra of  dimension $n (n>p+2)$ of degree $p$  over a field $\mathbb{F}$ characteristic zero. Then, $\mathcal{A}$ isomorphic to $\mu_0^{n-p}\oplus\mathbb{F}^p$.
 \end{teo}

\begin{de} An associative algebra $\mathcal{A}$ is called $p$-filiform if $C(\mathcal{A})=(n-p,1,\dots,1)$, where $p\geq0$.
\end{de}

Notice that when $p = 0$, then it coincides with null-filiform; if $p = 1$ coincides with filiform and in the case $p = 2$, it is a particular case of quasi-filiform. Moreover filiform associative algebra of degree $p$ is also a particular case of $p$-filiform associative algebra.

\section{Main result} \label{S:fil}

Let $\mathcal{A}$  be a naturally graded $n$-dimensional $p$-filiform associative algebra. Then there exists
a basis $\{e_1,e_2,\dots, e_{n-p}, f_1,\dots, f_p\}$ such that $e_1 \in \mathcal{A}\backslash\mathcal{A}^2$ and $C(e_1)=(n-p,1,\dots,1)$.

By the definition of characteristic sequence the operator $L_{e_1}$ in the Jordan form has one block $J_{n-p}$ of size $n-p$ and $p$ blocks $J_1$ of size one.

The possible forms for the operator $L_{e_1}$ are the following:
\[\begin{pmatrix}
J_{n-p}&0&\dots&0\\
0&J_1&\dots&0\\
\vdots&\vdots&\ddots&\vdots\\
0&0&\dots&J_1
\end{pmatrix},\quad
\begin{pmatrix}
J_1&0&\dots&0\\
0&J_{n-p}&\dots&0\\
\vdots&\vdots&\ddots&\vdots\\
0&0&\dots&J_1
\end{pmatrix},\dots,
\begin{pmatrix}
J_1&0&\dots&0\\
0&J_1&\dots&0\\
\vdots&\vdots&\ddots&\vdots\\
0&0&\dots&J_{n-p}
\end{pmatrix}.\]

By a shift of the basic elements, it is easy to prove that all cases when the Jordan block $J_{n-p}$ is placed in a position from the first, are mutually isomorphic cases. Thus, there are two possibilities of Jordan form of the matrix $L_{e_1}$:
\[\begin{pmatrix}
J_{n-p}&0&\dots&0\\
0&J_1&\dots&0\\
\vdots&\vdots&\ddots&\vdots\\
0&0&\dots&J_1
\end{pmatrix},\quad
\begin{pmatrix}
J_1&0&\dots&0\\
0&J_{n-p}&\dots&0\\
\vdots&\vdots&\ddots&\vdots\\
0&0&\dots&J_1
\end{pmatrix}.\]

Let us suppose that the operator $L_{e_1}$ has the second form. Then there exists
a basis $\{e_1,e_2,\dots, e_{n-p}, f_1,\dots, f_p\}$  and we have the next multiplications:
\[\left\{\begin{array}{ll}
e_1e_1=0,\\
e_1e_i=e_{i+1}, & 2\leq i\leq n-p-1,\\
e_1e_{n-p}=f_1,\\
e_1f_j=0, & 1\leq j\leq p.
\end{array}\right. \]

From the chain of equalities:
\[f_1=e_1e_{n-p}=e_1(e_1e_{n-p-1})=(e_1e_1)e_{n-p-1}=0\]
we obtain contradiction.

Thus, we can reduce the study to the following form of the matrix $L_{e_1}$:
\[\begin{pmatrix}
J_{n-p}&0&\dots&0\\
0&J_1&\dots&0\\
\vdots&\vdots&\ddots&\vdots\\
0&0&\dots&J_1
\end{pmatrix}\]

Let $\mathcal{A}$ be a naturally graded p-filiform associative algebra. Then there exists
a basis $\{e_1,e_2,\dots, e_{n-p}, f_1,\dots, f_p\}$ (so-called adapted basis) such that

\[\left\{\begin{array}{ll}
e_1e_i=e_{i+1}, & 1\leq i\leq n-p-1,\\
e_1f_j=0, & 1\leq j\leq p.
\end{array}\right. \]
By induction and associativity low, from multiplications given above it is easy to obtain that
\[\left\{\begin{array}{ll}
e_ie_j=e_{i+j}, & 2\leq i+j\leq n-p,\\
e_kf_s=0, & 1\leq k\leq n-p, 1\leq s\leq p.
\end{array}\right. \]

From these products, we have
\[\langle e_i\rangle \subseteq \mathcal{A}_i, \quad 1\leq i\leq n-p.\]

However, we do not have information about the elements $\{f_1,f_2,f_3,\dots,f_p\}$.

Let us denote by $\{r_1,r_2,r_3,\dots,r_p\}$, the position of the basic elements \(\{f_1,f_2,f_3,\dots,f_p\}\),
respectively, in natural gradation, i.e., $f_i\in \mathcal{A}_{r_i}$  for $1 \leq i \leq p$. Without loss of
generality, one can suppose that $r_1\leq r_2 \leq r_3 \leq\dots \leq r_p\leq n-p$.  For $p$-filiform associative algebras, the following
theorem holds.

\begin{teo}Let $\mathcal{A}$ be a graded p-filiform associative algebra. Then $r_s \leq s$
for any $s\in\{1,2,\dots,p\}$.
\end{teo}

\begin{proof}
It should be noted that $r_1 = 1$. Indeed, if $r_1 > 1$, then the algebra $\mathcal{A}$ has one generator and  it is null-filiform associative algebra and hence $C(\mathcal{A})=(n)$ we obtain a contradiction with our hypothesis.

We shall prove that $r_2 \leq 2$. Let us suppose the opposite, i.e., $r_2 > 2$. Then
\[\mathcal{A}_1=\langle e_1,f_1\rangle\, \quad \mathcal{A}_2=\langle e_2\rangle \quad \dots \quad \mathcal{A}_{r_2-1}=\langle e_{r_2-1}\rangle,\]
\[\mathcal{A}_{r_2}=\mathcal{A}_1\mathcal{A}_{r_2-1}=\langle e_1, f_1\rangle\cdot\langle e_{r_2-1}\rangle=\langle e_{r_2}, f_1e_{r_2-1}\rangle.\]

Consider the equations
\[f_1e_{r_2-1}=f_1(e_{r_2-2}e_1)=(f_1e_{r_2-2})e_1=\alpha e_{r_2-1}e_1=\alpha e_{r_2}.\]

Since $\mathcal{A}_{r_2}=\langle e_{r_2}\rangle$. So, $f_2\notin \mathcal{A}_{r_2},$ and we obtain a contradiction, hence $r_2\leq 2$.

By induction on $s$ we prove that $r_s \leq s$. Let us assume that $r_k\leq k$ where  $1\leq k\leq s-1$. Suppose that $r_s > s$. First we show that
\[f_te_{r_s-r_t}\subseteq\langle e_{r_s}\rangle, \quad 1\leq t\leq s-1.\]

If $t=s-1$ then,
\[f_{s-1}e_{r_s-r_{s-1}}=f_{s-1}(e_1e_{r_s-r_{s-1}-1})=(f_{s-1}e_1)e_{r_s-r_{s-1}-1}=\alpha e_{r_{s-1}+1}e_{r_s-r_{s-1}-1}=\alpha e_{r_s},\]
which implies $f_{s-1}e_{r_s-r_{s-1}}\in\langle e_{r_s}\rangle$.

So we have to show $f_te_{r_s-r_t}\subseteq\langle e_{r_s}\rangle, \ 1\leq t\leq s-2$. Since, $f_te_{r_s-r_t}=f_t(e_1e_{r_s-r_t-1})=(f_te_1)e_{r_s-r_t-1}$ and $f_te_1\in \mathcal{A}_{r_t}\mathcal{A}_1\subseteq\mathcal{A}_{r_t+1}$,
it is necessary distinguish the following cases:

\begin{itemize}
  \item \textbf{Case 1.} If $r_t+1=r_{t+1}$ then we have $\mathcal{A}_{r_t+1}=\mathcal{A}_{r_{t+1}}=\langle e_{r_{t+1}}, f_{t+1}\rangle$. From this we conclude that $f_{t+1}e_{r_{s}-r_{t}-1}=f_{t+1}e_{r_{s}-r_{t+1}}\in\langle e_{r_{s}}\rangle$, hence that
      \[f_te_{r_s-r_t}=f_t(e_1e_{r_s-r_t-1})=(f_te_1)e_{r_s-r_t-1}=(\alpha e_{r_{t+1}}+\beta f_{t+1})e_{r_s-r_t-1}\in\langle e_{r_{s}}\rangle.\]
  \item \textbf{Case 2.} If $r_t+1>r_{t+1}$ then $f_te_1\in\langle e_{r_t+1}\rangle=\mathcal{A}_{r_t+1}$ and by using similar arguments than
above case we obtain $f_te_{r_s-r_t}\in\langle e_{r_s}\rangle$.
\end{itemize}

Thus the embedding $f_te_{r_s-r_t}\in\langle e_{r_s}\rangle, 1\leq t\leq s-1$, is proved.

Let us prove that $\mathcal{A}_{r_s}\subseteq\langle e_{r_s}\rangle$ under the assumption $r_s > s$.

Now we consider the product
\[\mathcal{A}_{r_s}=\mathcal{A}_1\mathcal{A}_{r_s-1}=\langle e_1, f_i\rangle\cdot\langle e_{r_s-1}\rangle=\langle e_{r_s},f_ie_{r_s-1}\rangle\]
for some $1\leq i\leq s-1$. The embedding $f_te_{r_s-r_t}\in\langle e_{r_s}\rangle$ deduces $\mathcal{A}_{r_s}\subseteq\langle e_{r_s}\rangle$, i.e., we obtain a contradiction with the assumption $r_s>s$ which completes the proof of the theorem.
\end{proof}

\textit{Remark.} An arbitrary $p$-filiform associative algebra satisfying the property $r_i = 1$ for $1 \leq i \leq p$ is a naturally graded filiform associative algebra of degree $p$ which classified in \cite{kar3}.

Let us denote that $\mathcal{A}_1=\langle e_1,f_1,\dots,f_{s_1}\rangle$ and $\mathcal{A}_i=\langle e_i,f_{s_1+\dots+s_{i-1}+1},\dots,f_{s_1+\dots+s_i}\rangle$ for $2\leq i\leq n-p$, where $s_1,\dots,s_{n-p}$ are natural number such that $s_1+\dots+s_{n-p}=p$ and $dim\mathcal{A}_i=s_i+1$ for $1\leq i\leq n-p$.

\begin{teo}Let $\mathcal{A}$ is $n$ dimensional naturally graded p-filiform associative algebra. Then $0\leq s_{n-p}\leq \dots\leq s_2\leq s_1<p$.
\end{teo}

\begin{proof} Since $\mathcal{A}_1=\langle e_1,f_1,\dots,f_{s_1}\rangle, \mathcal{A}_2=\langle e_2,f_{s_1+1},\dots,f_{s_1+s_2}\rangle$ and $\mathcal{A}_1\mathcal{A}_1\subseteq\mathcal{A}_2$  we denote
\[f_ie_1=a_ie_2+\sum\limits_{j=1}^{s_2}a_{ij}f_{s_1+j}, \quad f_if_j=b_{ij}e_2+\sum\limits_{k=1}^{s_2}b_{ijk}f_{s_1+k}, \quad 1\leq i,j\leq s_1.\]

Then, from the equalities:
\[\begin{array}{l}
0=(e_1f_i)e_1=e_1(f_ie_1)=a_ie_3+\sum\limits_{j=1}^{s_2}a_{ij}e_1f_{s_1+j}=a_ie_3,\\
0=(e_1f_i)f_j=e_1(f_if_j)=b_{ij}e_3+\sum\limits_{k=1}^{s_2}b_{ijk}e_1f_{s_1+j}=b_{ij}e_3,
\end{array}\]
we obtain
\[f_ie_1=\sum\limits_{j=1}^{s_2}a_{ij}f_{s_1+j}, \quad f_if_j=\sum\limits_{k=1}^{s_2}b_{ijk}f_{s_1+k}, \quad 1\leq i,j\leq s_1.\]

We make the following general transformation of basis for $1\leq i\leq s_1$:
\[e_1^\prime=e_1+Af_i(A\neq0), \quad e_j^\prime=e_1^\prime e_{j-1}^\prime, \quad  2\leq j\leq n-p. \]
Note that $e_1^\prime e_{n-p}^\prime=0$ and
\[e_1^\prime f_j=(e_1+Af_i)f_j=Af_if_j=A\sum\limits_{k=1}^{s_2}b_{ijk}f_{s_1+k}, \quad 1\leq i,j\leq s_1.\]
If $b_{ijk}\neq0$, then the rank of $L_{e_1^\prime}$ would be greater than $n-p-1$, which contradicts the assumption
$C(\mathcal{A}) = (n-p,1,\dots, 1)$. Hence,
\[f_if_j=0, \quad 1\leq i,j\leq s_1.\]
Analogously, we can deduce
\[f_if_j=f_jf_i=0, \quad f_if_t, f_tf_i\in\langle e_{n-p}\rangle \]
for $1\leq i\leq s_1, s_1+\dots+s_{k-1}+1\leq j\leq s_1+\dots+s_k,k\neq n-p-1, s_1+\dots+s_{n-p-2}+1\leq t\leq s_1+\dots+s_{n-p-1}.$

So, we have $\mathcal{A}_2=\mathcal{A}_1\mathcal{A}_1=\langle e_2, f_1e_1,\dots,f_{s_1}e_1\rangle\subseteq\langle e_2,f_{s_1+1},\dots,f_{s_1+s_2}\rangle$ therefore, $dim\mathcal{A}_2\leq dim\mathcal{A}_1=s_1+1$ and $s_2\leq s_1$.

Now we show that $s_1+\dots+s_k<p$ yields $s_{k+1}\neq0$. Let us suppose that $s_{k+1}=0$ then $\mathcal{A}_{k+1}=\langle e_{k+1}\rangle$. As $\mathcal{A}_{k+1}=\mathcal{A}_1\mathcal{A}_k$ we deduce $f_ie_k=\alpha_ie_{k+1}$ for $1\leq i\leq s_1$. From \[\alpha_ie_{k+2}=\alpha_ie_1e_{k+1}=e_1(f_ie_k)=(e_1f_i)e_k=0\]
we obtain $\alpha_i=0$ for  $1\leq i\leq s_1$. We have $f_ie_k=0$ for $1\leq i\leq s_1$. The chain equalities \[f_ie_{k+1}=f_i(e_ke_1)=(f_ie_k)e_1=0\] implies that $f_ie_j=0$ for $1\leq i\leq s_1$ and $k\leq j\leq n-p-1.$ Thus $\mathcal{A}_{l}=\langle e_l\rangle$ for $k+2\leq l\leq n-p$. However, if we had these equalities the vector $f_p$ would not be obtained and its implies that $s_{k+1}\neq0$.

Since $\mathcal{A}_{i+1}=\mathcal{A}_1\mathcal{A}_i=\langle e_{i+1}, f_1e_i,\dots,f_{s_1}e_i\rangle$ and $f_je_i=(f_je_{i-1})e_1$ for $1\leq j\leq s_1$, implies that $dim(\langle f_1e_i,\dots,f_{s_1}e_i\rangle)\leq dim(\langle f_1e_{i-1},\dots,f_{s_1}e_{i-1}\rangle)$, follows $dim\mathcal{A}_{i+1}\leq dim\mathcal{A}_{i}$ for $1\leq i\leq n-p-1$. Hence, it is proved that $0\leq s_{n-p}\leq \dots\leq s_2\leq s_1<p$.

\end{proof}

\begin{lemma} Let $\mathcal{A}$ is $n$ dimensional naturally graded p-filiform associative algebra under the above assumptions. Then
\[f_{s_1+\dots+s_k+i}e_j=\left\{\begin{array}{cl}
 f_{s_1+\dots+s_{k+j}+i},  & 1\leq i\leq s_{k+j+1}, \\
 0,  & s_{k+j+1}+1\leq i\leq s_{k+j} \\
\end{array}\right.\]
for $1\leq j\leq n-p-k, 1\leq k \leq n-p-2.$
\end{lemma}
\begin{proof} Since
\[\langle e_{l+2}, f_{s_1+\dots+s_{l+1}+1},\dots, f_{s_1+\dots+s_{l+2}}\rangle=\mathcal{A}_{l+2}\supseteq\mathcal{A}_{l+1}\mathcal{A}_{1}=\]\[\langle e_{l+2}, f_{s_1+\dots+s_l+1}e_1,\dots, f_{s_1+\dots+s_{l+1}}e_1\rangle,\] there exist $1\leq m_1, m_2, \dots, m_{s_{l+2}}\leq s_{l+1}$ such that \[\langle f_{s_1+\dots+s_l+m_1}e_1,\dots, f_{s_1+\dots+s_l+m_{s_l+2}}e_1\rangle=\langle f_{s_1+\dots+s_{l+1}+1},\dots,f_{s_1+\dots+s_{l+2}}\rangle.\]

By making the next basis transformations
\[\begin{array}{llll}
f_j^\prime=f_{m_j}, & f_{s_1+j}^\prime=f_{s_1+m_j}, &  \dots, & f_{s_1+\dots+s_l+j}^\prime=f_{s_1+\dots+s_l+m_j},\\
 f_{m_j}^\prime=f_j, & f_{s_1+m_j}^\prime=f_{s_1+j}, &\dots, & f_{s_1+\dots+s_l+m_j}^\prime=f_{s_1+\dots+s_l+j}
\end{array}\]
for $1\leq j\leq s_{i+1},$ we get
\[f_{s_1+\dots+s_{k-1}+i}^\prime e_1=f_{s_1+\dots+s_k+i}, \quad 1\leq i\leq s_{k+1}, 1\leq k\leq l.\]
Thus $\langle f_{s_1+\dots+s_l+1}^\prime e_1, \dots, f_{s_1+\dots+s_{l+1}}^\prime e_1\rangle=\langle f_{s_1+\dots+s_{l+1}+1},\dots f_{s_1+\dots+s_{l+2}}\rangle.$

Then subsequently $f_{s_1+\dots+s_{l+1}+i}^\prime=f_{s_1+\dots+s_l+i}^\prime e_1$ for $1\leq i\leq s_{l+1}$ we have
\[f_{s_1+\dots+s_{k-1}+i}^\prime e_1=f_{s_1+\dots+s_k+i}^\prime, \quad 1\leq i\leq s_{k+1}, \quad 1\leq k\leq l+1.\]

As $m_1,\dots,m_{s_l+2}<s_{l+1}$ the following zero multiplications stay unchanged:
\[f_{_1+\dots+s_{k-1}+s_{k+1}+i}e_1=0, \quad 1\leq i\leq s_k-s_{k+1}, \quad 1\leq k\leq l.\]

If \[f_{s_1+\dots+s_l+i}e_1=\sum\limits_{k=1}^{s_{l+2}}c_kf_{s_1+\dots+s_{l+1}+k}, \quad s_{l+2}+1\leq s_{l+1}\]
then the following change of basis
\[f_{s_1+\dots+s_l+i}^\prime=f_{s_1+\dots+s_l+i}-\sum\limits_{k=1}^{s_{l+2}}c_kf_{s_1+\dots+s_{l+1}+k}\]
implies that
\[f_{s_1+\dots+s_l+i}^\prime e_1=0, \quad s_{l+2}+1\leq i\leq s_{l+1}.\]
Following we have
\[f_{s_1+\dots+s_{t-1}+i}e_1=\left\{\begin{array}{cl}
 f_{s_1+\dots+s_{t}+i},  & 1\leq i\leq s_{t+1}, \\
 0,  & s_{t+1}+1\leq i\leq s_{t} \\
\end{array}\right.\]
for $1\leq t\leq l$ and $1\leq l\leq n-p-2.$

Now we shall prove the following expression by an induction on $j:$
\[f_{s_1+\dots+s_k+i}e_j=\left\{\begin{array}{cl}
 f_{s_1+\dots+s_{k+j}+i},  & 1\leq i\leq s_{k+j+1}, \\
 0,  & s_{k+j+1}+1\leq i\leq s_{k+j} \\
\end{array}\right.\]
for $1\leq j\leq n-p-k.$

Obviously, the expression holds for $j=1$. Let us assume that the equality holds for  $1\leq j\leq n-p-k$ where $2\leq k\leq n-p-2.$ Note that the multiplications $f_{s_1+\dots+s_k+i}e_j=0$ holds by property of graduation. We will prove it for $j+1:$
\[f_{s_1+\dots+s_k+i}e_{j+1}=f_{s_1+\dots+s_k+i}(e_{j}e_1)=(f_{s_1+\dots+s_k+i}e_{j})e_1=\]\[f_{s_1+\dots+s_{k+j}+i}e_1=f_{s_1+\dots+s_{k+j+1}+i}.\]
The lemma is proved
\end{proof}

We consider
\[f_{s_1+\dots+s_t+i}f_j=(f_{s_1+\dots+s_{t-1}+i}e_1)f_j=f_{s_1+\dots+s_{t-1}+i}(e_1f_j)=0\]
and
\[f_jf_{s_1+\dots+s_t+i}=f_j(f_ie_t)=(f_jf_i)e_t=0\]
for $1\leq i\leq s_{t+2}, 1\leq j\leq p.$

Let us denote
\[f_{s_1+\dots+s_k+i}f_{s_1+\dots+s_t+j}=b_{i,j}^{k,t}e_{k+t+2}+\sum\limits_{l=1}^{s_{k+t+2}}b_{i,j,l}^{k,t}f_{s_1+\dots+s_{k+t+1}+l}\]
where $ 1\leq k \leq n-p-3,  1\leq t\leq n-p-2-k, s_{k+2}+1\leq i\leq s_{k+1}$ and $s_{t+2}+1\leq j\leq s_{t+1}$.

Considering the next identity:
\[0=(e_1f_{s_1+\dots+s_k+i})f_{s_1+\dots+s_t+j}=e_1(f_{s_1+\dots+s_k+i}f_{s_1+\dots+s_t+j})=b_{ij}^{kt}e_1e_{k+t+2}=b_{i,j}^{k,t}e_{k+t+3}\]
we obtain that:
\[f_{s_1+\dots+s_k+i}f_{s_1+\dots+s_t+j}=\sum\limits_{l=1}^{s_{k+t+2}}b_{i,j,l}^{k,t}f_{s_1+\dots+s_{k+t+1}+l}\]
where $1\leq k \leq n-p-3,  1\leq t\leq n-p-2-k, s_{k+2}+1\leq i\leq s_{k+1}$ and $s_{t+2}+1\leq j\leq s_{t+1}$.

Considering the next identity:

Let $1\leq k \leq n-p-4,  1\leq t\leq n-p-3-k, s_{k+2}+1\leq i\leq s_{k+1}$ and $s_{t+2}+1\leq j\leq s_{t+1}$.

\[(f_{s_1+\dots+s_k+i}f_{s_1+\dots+s_t+j})e_1=
\sum\limits_{l=1}^{s_{k+t+3}}b_{i,j,l}^{k,t+1}f_{s_1+\dots+s_{k+t+1}+l}e_1=\sum\limits_{l=1}^{s_{k+t+3}}b_{i,j,l}^{k,t+1}f_{s_1+\dots+s_{k+t+2}+l},\]

\[(f_{s_1+\dots+s_k+i}f_{s_1+\dots+s_t+j})e_1=
f_{s_1+\dots+s_k+i}(f_{s_1+\dots+s_t+j}e_1)=0.\]

\[(f_{s_1+\dots+s_k+i}f_{s_1+\dots+s_t+j})e_1=
f_{s_1+\dots+s_k+i}(f_{s_1+\dots+s_t+j}e_1)=0.\]

We have
\[f_{s_1+\dots+s_k+i}f_{s_1+\dots+s_t+j}=0,\]
where $ 1\leq k \leq n-p-4,  1\leq t\leq n-p-3-k, s_{k+2}+1\leq i\leq s_{k+1}$ and $s_{t+2}+1\leq j\leq s_{t+1}$.

Moreover
\[f_{s_1+\dots+s_k+i}f_{s_1+\dots+s_t+j}=b_{i,j}^{k,t}e_{n-p}+\sum\limits_{l=1}^{s_{n-p}}b_{i,j,l}^{k,t}f_{s_1+\dots+s_{n-p-1}+l}\]
where $k+t=n-p-2, kt\neq0, s_{k+2}+1\leq i\leq s_{k+1}$ and $s_{t+2}+1\leq j\leq s_{t+1}$.

\begin{teo} Let $\mathcal{A}$ is $n$ dimensional naturally graded p-filiform associative algebra under the above assumptions. Then
\[\begin{array}{ll}
e_ie_j=e_{i+j}, & 2\leq i+j\leq n-p,\\
f_{s_1+\dots+s_k+i}e_j=f_{s_1+\dots+s_{k+j}+i},  & 1\leq k \leq n-p-2, 1\leq j\leq n-p-k, 1\leq i\leq s_{k+j+1},\\
\end{array} \]

\[f_{s_1+\dots+s_k+i}f_{s_1+\dots+s_t+j}=b_{i,j}^{k,t}e_{n-p}+\sum\limits_{l=1}^{s_{n-p}}b_{i,j,l}^{k,t}f_{s_1+\dots+s_{n-p-1}+l}\]
where $k+t=n-p-2, kt\neq0, s_{k+2}+1\leq i\leq s_{k+1}$ and $s_{t+2}+1\leq j\leq s_{t+1}$.
Other multiplications are zero.
\end{teo}

\textbf{Acknowledgement.} We thank K.Abdurasulov for the helpful comments and suggestions that contributed to improving this paper.

\end{document}